\documentclass[12pt]{amsart}
\usepackage{amssymb, amsthm, enumerate, graphicx, amscd, amsmath, amssymb, amsfonts, xcolor, tikz, tikz-cd, tkz-euclide, comment}
\usepackage[UKenglish]{babel}
\usepackage[mathcal]{euscript}
\numberwithin{equation}{section}
\newtheorem{theorem}[equation]{Theorem}
\newtheorem{lemma}[equation]{Lemma}

\newtheorem{proposition}[equation]{Proposition}

\newtheorem{example}[equation]{Example}

\begin{document}
\title[]{The spectral analysis of the difference quotient operator on model spaces}

\author[C. Bellavita]{Carlo Bellavita}
\email{carlo.bellavita@gmail.com}
\address{Departament of Matematica i Informatica, Universitat de Barcelona, Gran Via 585, 08007 Barcelona, Spain.}

\author[E.A. Dellepiane]{Eugenio A. Dellepiane}
\email{eugenio.dellepiane@unimi.it}
\address{Department of Mathematics, Università degli studi di Milano, Via C. Saldini 50, 20133, Italy.}

\author[J. Mashreghi]{Javad Mashreghi}
\email{javad.mashreghi@mat.ulaval.ca}
\address{D\'epartement de math\'ematiques et de statistique, Universit\'e Laval, Qu\'ebec, QC, Canada G1K 7P4.}

\date{\today}

\thanks{This work was supported by grants NSERC and the Canada Research Chairs program. The first author and the second author are members of Gruppo Nazionale per l'Analisi Matematica, la Probabilità e le loro Applicazioni (GNAMPA) of Istituto Nazionale di Alta Matematica (INdAM). The first author was partially supported by PID2021-123405NB-I00 by the Ministerio de Ciencia e Innovación and by the Departament de Recerca i Universitats, grant 2021 SGR 00087.}

\begin{abstract}
We conduct a spectral analysis of the difference quotient operator $\text{Q}_\zeta^u$, associated with a boundary point $\zeta \in \partial \mathbb{D}$, on the model space $K_u$. We describe the operator's spectrum and provide both upper and lower estimates for its norm $\|\text{Q}_\zeta^u\|$, and furthermore discussing the sharpness of these bounds. Notably, the upper estimate offers a new characterization of the one-component property for inner functions.
\end{abstract}

\maketitle

\section{Introduction} \label{S:Introduction}
Let $\mathcal{X}$ be a family of functions defined on the set $\Omega \subset \mathbb{C}$, and fix a point $w \in \Omega$. The difference quotient operator on $\mathcal{X}$ is given by
\begin{equation}\label{E:def-Qw}
(\text{Q}_w f)(z) := \frac{f(z) - f(w)}{z-w}, \qquad z \in \Omega \setminus \{w\}, \quad f\in\mathcal{X}.  
\end{equation}
This operator appears in many discussions and exhibits numerous valuable properties. In particular, for the Hardy space on the unit disk, $H^2(\mathbb{D})$, the special but important case $w=0$ leads to the backward shift operator
\[
S^\ast f(z) := \text{Q}_0 f(z) = \frac{f(z) - f(0)}{z}, \qquad z \in \mathbb{D} \setminus \{0\},
\]
which is one of the most studied operators in all of operator theory \cite{MR827223}. In fact, for any $w \in \mathbb{D}$, the operator $\text{Q}_w$ is bounded on $H^2(\mathbb{D})$, and $\text{Q}_w$ and $S^*$ are related via the equation
\begin{equation}\label{E:relation-Q-S}
\text{Q}_w = (\text{I} - w S^*)^{-1} S^{*}, \qquad w \in \mathbb{D}.
\end{equation}

A natural obstacle arises when we move away from interior points of the disk $\mathbb{D}$, and instead consider difference quotients at a boundary point $\zeta \in \mathbb{T} := \partial \mathbb{D}$. We need to clarify the meaning of $f(\zeta)$. For example, if the ambient space is the whole Hardy space $H^2(\mathbb{D})$, this goal cannot be achieved, since there are functions $f \in H^2(\mathbb{D})$ for which the radial limit $f(\zeta):= \lim_{r \to 1} f(r\zeta)$ does not exist. However, when restricted to a proper subclass of $H^2(\mathbb{D})$, the situation could change. In this paper, we study the difference quotient operator associated with points on the unit circle $\mathbb{T}$, defined on model spaces $K_u$, which form an important family of subspaces of $H^2(\mathbb{D})$.

Model spaces arise as the closed $S^*$-invariant subspaces of $H^2(\mathbb{D})$. More explicitly, the model space $K_u$ is the orthogonal complement of the Beurling invariant subspace $uH^2(\mathbb{D})$, i.e.,
\[
K_u := H^2(\mathbb{D}) \ominus u H^2(\mathbb{D}),
\]
endowed with the $H^2$-norm. Here, $u$ is an inner function, i.e., a bounded holomorphic function on $\mathbb{D}$ such that $\lim_{r \to 1}|u(r\zeta)| = 1$ for $m$-a.e. $\zeta \in \mathbb{T}$, where $m$ is the normalized Lebesgue measure on $\mathbb{T}$. A detailed description of model spaces can be found in \cite{cima2000backward} and \cite{MR3526203}. By restricting $S^*$ to $K_u$, we obtain
\[
\begin{array}{cccc}
X_u : & K_u & \longrightarrow & K_u \\
& f & \longmapsto & S^{*}f
\end{array}
\]
which is a bounded operator on $K_u$. We also define $S_u := X_u^*$, the so-called compressed shift on $K_u$ \cite[Ch. 9]{MR3526203}. One of the fundamental results regarding the compressed shift is the Liv\v{s}ic--M\"{o}ller theorem \cite{MR0113142, MR0150592}, which describes the spectrum of $S_u$ to be precisely
\[ \sigma(S_u)= \sigma(u) := \{ \lambda \in \overline{\mathbb{D}} : \liminf_{z \to \lambda} |u(z)| = 0\},\]
where the latter is referred to as the spectrum of $u$. We also refer to the intersection $\sigma(u) \cap \mathbb{T}$ as the boundary spectrum of $u$. The spectrum $\sigma(u)$ plays a central role in our analysis, as the difference quotient operator $\text{Q}_\zeta^u$ given by \eqref{E:def-Qw} is well-defined and bounded on $K_u$ whenever $\zeta \in \mathbb{T} \setminus \sigma(u)$.

To study further different parts of the spectrum, recall that the \emph{point spectrum} of an operator $A \in B(H)$ is defined as \[
\sigma_p(A) := \{\lambda \in \mathbb{C} \colon \lambda \text{I} - A \,\, \text{is not injective} \},
\]
and its \emph{essential spectrum} is
\[
\sigma_e(A) := \{\lambda \in \mathbb{C} \colon \lambda \text{I} - A \,\, \text{is not a Fredholm operator} \}.
\]
An operator $A$ is said to be a \emph{Fredholm operator} if it has a closed range and both $\text{Ker}(A)$ and $\text{Ker}(A^*)$ are finite-dimensional. For the operator $X_u$, we know that
\begin{equation} \label{E:spectraXu1}
\sigma_p(X_u) = \left\{ \lambda \in \mathbb{D} \colon \overline{\lambda} \in \sigma(u) \right\},
\end{equation}
and
\begin{equation}\label{E:spectraXu2}
\sigma_e(X_u) = \left\{ \lambda \in \mathbb{T} \colon \overline{\lambda} \in \sigma(u) \right\}.
\end{equation}
The identity in \eqref{E:spectraXu1} corresponds to Equation $(9.15)$ in \cite{MR3526203}, while \eqref{E:spectraXu2} follows from the description of $\sigma_e(S_u)$ \cite[Proposition $9.26$]{MR3526203}, along with the fact that an operator $A$ is Fredholm if and only if its adjoint $A^*$ is also Fredholm.

For the remainder of this paper, we will consider inner functions $u$ for which the boundary spectrum is not the entire unit circle, i.e., $\sigma(u) \cap \mathbb{T} \neq \mathbb{T}$. We will also denote by $\| \cdot \|$ the operator norm in the Banach algebra $B(K_u)$ of bounded linear operators on $K_u$.

The organization of the paper is as follows. In Section \ref{S:main-results}, the main results, which are Theorems \ref{T:spectrumQ}, \ref{T:lower-estimate}, \ref{T:upper-estimate}, and \ref{T:characterization}, are stated. 

In Section \ref{S:Preliminaries}, we present some preliminary notions as well as some profound results about model spaces, the difference quotient operator, and one-component inner functions. In Section \ref{S:Lowerestimates}, we prove Theorem \ref{T:spectrumQ} and describe in more detail the structure of the spectrum $\sigma(\text{Q}_\zeta^u)$. In Section \ref{S:Lowerestimates2}, we first establish a technical lemma and then use it to prove Theorem \ref{T:lower-estimate}. In Section \ref{S:Upperestimates}, we prove Theorem \ref{T:upper-estimate}. In Section \ref{S:Example0}, we prove Theorem \ref{T:characterization}. In Section \ref{S:Example}, we present an example showing that the exponents of the derivatives involved in Theorems \ref{T:lower-estimate} and \ref{T:upper-estimate} are sharp. Finally in Section \ref{S:Example2}, as concluding remarks, we discuss a connection between $\text{Q}_\zeta^u$ and local Dirichlet spaces, and present a formula for the norm of the resolvent operator of the compressed shift $S_u$, in terms of the norm $\|\text{Q}_\zeta^u\|$.

\section{Main results} \label{S:main-results}
In our first result, we describe the spectrum of $\text{Q}_\zeta^u$ in terms of the spectrum of $u$ and we deduce a lower bound for its operator norm.

\begin{theorem}\label{T:spectrumQ}
Let $u$ be an inner function and $\zeta \in \mathbb{T} \setminus \sigma(u)$. Then 
\begin{eqnarray}
\sigma(\text{Q}_\zeta^u) &=& \left\{\frac{\overline{\eta}}{1-\zeta \overline{\eta}} \colon \eta\in\sigma(u)\right\}, \label{E:spectrumQ}\\
\sigma_p(\text{Q}_\zeta^u) &=& \left\{\frac{\overline{\eta}}{1-\zeta \overline{\eta}} \colon \eta\in\sigma(u)\cap \mathbb{D}\right\}, \label{E:spectraQz1}  \\
\sigma_e(\text{Q}_\zeta^u) &=& \left\{\frac{\overline{\eta}}{1-\zeta \overline{\eta}} \colon \eta\in\sigma(u)\cap \mathbb{T}\right\}. \label{E:spectraQz2} 
\end{eqnarray}
Moreover, the lower estimate
\begin{equation}\label{E:spectrumQest}
\|\text{Q}_\zeta^u\| \geq \frac{1}{\operatorname{dist}\left(\zeta,\sigma(u)\cap\mathbb{T}\right)}
\end{equation}
holds.
\end{theorem}

We also prove another lower estimate, which turns out to be sharper than \eqref{E:spectrumQest} in some special cases.

\begin{theorem} \label{T:lower-estimate}
Let $u$ be an inner function and $\zeta\in\mathbb{T}\setminus \sigma(u)$. Then,
\begin{equation}\label{E:lower-estimate}
\|\text{Q}_\zeta^u\| \geq \frac{\, |u''(\zeta)| \,}{\, 2|u'(\zeta)| \,}.
\end{equation}
\end{theorem}

The above two results are for arbitrary inner functions. By restricting to the special class of \emph{one-component} inner functions, an important subclass of inner functions, we obtain further qualitative results. In the following, we provide an upper bound for $\|\text{Q}_\zeta^u\|$ when $u$ belongs to this class.

\begin{theorem}\label{T:upper-estimate}
Let $u$ be a one-component inner function, and let $\zeta\in\mathbb{T}\setminus\sigma(u)$. Then
\begin{equation}\label{E:upper-estimate}
\|\text{Q}_\zeta^u \| \leq C_u |u'(\zeta)|,
\end{equation}
where $C_u$ is a positive quantity not depending on $\zeta$.
\end{theorem}

In fact, Theorem $\ref{T:upper-estimate}$ provides a characterization of the property of being one-component. Our final result directly addresses this characterization. Recall that, according to Aleksandrov's theorem (Proposition \ref{T:aleksandrov1}), $m(\sigma(u) \cap \mathbb{T})=0$ for any one-component inner function $u$.

\begin{theorem}\label{T:characterization}
Let $u$ be an inner function. Then $u$ is one-component if and only if $m(\sigma(u) \cap \mathbb{T})=0$ and  \eqref{E:upper-estimate} holds for every $\zeta \in \mathbb{T}\setminus \sigma(u)$.
\end{theorem}

More explicitly, Theorem \ref{T:characterization} says that $u$ is one-component if and only if $m(\sigma(u))=0$ and there exists a positive constant $C_u$, depending only on $u$, such that
\[
\int_{\mathbb{T}} \left|\frac{f(\lambda)-f(\zeta)}{\lambda - \zeta}\right|^2\, dm(\lambda) \leq C_u |u'(\zeta)|^2 \|f\|_{H^2}^2, \qquad f \in K_u.
\]

\section{Some profound results about model spaces} \label{S:Preliminaries}
In this section, we revisit some key properties of model spaces and introduce the difference quotient operator. These are foundational theorems in the theory; however, to avoid confusion with our own results, we refer to them as propositions. We begin by recalling a well-known result on angular derivatives \cite[Thm $7.24$]{MR3526203}.

\begin{proposition}\label{T:thm724model}
Let $u$ be an inner function and $\zeta\in\mathbb{T}$. Then the following are equivalent.
\begin{enumerate}[(i)]
\item The quantity
\[
c:=\liminf_{z\in\mathbb{D},z\to\zeta}\frac{1-|u(z)|}{1-|z|}
\]
is finite.
\item There exists $\lambda\in\mathbb{T}$ such that the function
\[
z\mapsto \frac{u(z)-\lambda}{z-\zeta}, \qquad z\in\mathbb{D},
\]
belongs to $K_u$.
\item Each $f\in K_u$ has non-tangential limit at $\zeta$.
\item The derivative $u'$ has non-tangential limit at $\zeta$.
\end{enumerate}
\end{proposition}
If the conditions of Proposition \ref{T:thm724model} are satisfied, further conclusions can be drawn about the inner function $u$. Specifically, we have $\lambda = u(\zeta)$ and $u'(\zeta)=\overline{\zeta}u(\zeta)|u'(\zeta)|$. Additionally, the \emph{boundary kernel}
\[
k_\zeta^u(z):=\frac{1-\overline{u(\zeta)}u(z)}{1-\overline{\zeta}z}
\]
belongs to $K_u$, and for each $f\in K_u$, 
\[
f(\zeta) = \langle f,k_\zeta^u \rangle.
\]
Finally, we have the relation
\[
c=\angle\lim_{z\to\zeta}\frac{1-|u(z)|}{1-|z|} = |u'(\zeta)| = k_\zeta^u(\zeta) = \|k_\zeta^u\|_{H^2}^2,
\]
where $\angle$ denotes the non-tangential limit.

We also require the following result, which establishes a connection between the boundary regularity of $u$ and the boundary regularity of the elements in $K_u$ \cite[Theorem 20.13]{MR3617311}.

\begin{proposition} \label{T:resolvent-u}
Let $\Delta$ be an open arc of $\mathbb{T}$. Then the following are equivalent.
\begin{enumerate}[(i)]
\item  $\Delta$ is contained in the resolvent of $S_u$, i.e., $\Delta \cap \sigma(S_u) = \emptyset$.
\item  $u$ has an analytic continuation across $\Delta$, with $|u|=1$ on $\Delta$.
\item  Every function in $K_u$ has an analytic continuation across $\Delta$.
\end{enumerate}
\end{proposition}

Equipped with Proposition \ref{T:resolvent-u}, we are now ready to define the operator $\text{Q}_\zeta^u$ properly. Fix a point $\zeta\in\mathbb{T}\setminus\sigma(u)$. Since $\sigma(S_u) = \sigma(u)$, by Proposition \ref{T:resolvent-u}, the operator
\[
\zeta\text{I} - S_u=\zeta\text{I} - X^*_u
\]
is invertible in the algebra $B(K_u)$ of bounded linear operators on $K_u$. Thus, the operator $\text{I} - \zeta X_u = \zeta(\zeta I - S_u)^*$ is also invertible, and, in light of \eqref{E:relation-Q-S}, we define
\begin{equation}\label{E:def-Qu}
\text{Q}_\zeta^u := (\text{I}-\zeta X_u)^{-1} X_u.
\end{equation}
In this case, for all $f\in K_u$, we also have
\[
(\text{Q}_\zeta^u f)(z) = \frac{f(z)-f(\zeta)}{z-\zeta}, \qquad z\in \mathbb{D},
\]
which is consistent with \eqref{E:def-Qw}.

An inner function $u$ is one-component if there exists an $\varepsilon \in (0,1)$ such that the sublevel set
\[
\Omega_\varepsilon := \{ z \in \mathbb{D} : |u(z)| < \varepsilon \}
\]
is connected. Other equivalent descriptions of one-component inner functions can be found in \cite{Aleksandrov2000, bessonov2015duality, nicolau2021}, and certain algebraic properties of this family have been explored in \cite{Cima2020, Reijonen_2019}. For a comprehensive survey, see \cite{Bellavita2022O}. In particular, the following characterization plays a crucial role in our work. Note that Proposition \ref{T:resolvent-u} is implicitly used in parts $(ii)$ and $(iii)$.

\begin{proposition}[Aleksandrov \cite{Aleksandrov2000}]\label{T:aleksandrov1}
An inner function $u$ is one-compo\-nent if and only if the following conditions hold.
\begin{enumerate}[(i)]
\item $m\big(\sigma(u) \cap \mathbb{T} \big)=0$.
\item $|u'|$ is unbounded on any open arc $\Delta\subset  \mathbb{T}\setminus \sigma(u)$ with $\overline{\Delta} \cap \sigma(u)\neq \emptyset$.
\item There exists a positive constant $C$ such that 
\[
|u''(\zeta)|\leq C |u'(\zeta)|^2, \qquad \zeta \in \mathbb{T} \setminus \sigma(u).
\]
\end{enumerate}
\end{proposition}

Using Aleksandrov's result, Bessonov characterized the Clark measures of the one-component inner functions \cite[Theorem $1$]{bessonov2015duality}. Given an inner function $u$ and $\alpha \in \mathbb{T}$, the Clark measure $\sigma_\alpha$ is the unique finite, positive, Borel measure on $\mathbb{T}$ such that
\[
\frac{1-|u(z)|^2}{|\alpha-u(z)|^2} = \int_{\mathbb{T}} \frac{1-|\zeta|^2}{|z-\zeta|^2} \, d\sigma_\alpha(\zeta), \qquad z \in \mathbb{D}.
\]
This topic is thoroughly discussed in \cite{Ross2013LENSLO, saksman2007}. Adopting Bessonov's notation, for every Borel measure $\mu$ on the unit circle $\mathbb{T}$, we denote the set of the isolated atoms of $\mu$ by $a(\mu)$ and define $\rho(\mu) := \text{supp}(\mu) \setminus a(\mu)$. We say that an atom $\zeta \in a(\mu)$ has two neighbors if there is an open arc $(\zeta_-, \zeta_+)$ of $\mathbb{T}$ with endpoints $\zeta_\pm \in a(\mu)$ such that $\zeta$ is the only point in $(\zeta_-, \zeta_+) \cap \text{supp}(\mu)$.

\begin{proposition}[Bessonov \cite{{bessonov2015duality}}]\label{T:Bessonov}
Let $u$ be a one-component inner function, and let $\alpha\in\mathbb{T}$. Then the associated Clark measure $\sigma_\alpha$ is a discrete measure on $\mathbb{T}$ with isolated atoms such that $m(\text{supp}(\sigma_\alpha)) = 0$. Moreover, every open arc $\Delta \subset \mathbb{T}\setminus \sigma(u)$ contains infinitely many atoms of $a(\sigma_\alpha)$ and every atom $\zeta \in a(\sigma_\alpha)$ has two neighbors $\zeta_\pm \in a(\sigma_\alpha)$ such that
\begin{equation}\label{T:BessonovEq1}
A_u |\zeta-\zeta_\pm|\leq \sigma_\alpha(\{\zeta\})\leq B_u |\zeta-\zeta_\pm|
\end{equation}
for some positive constants $A_u,B_u$ depending only on $u$.
\end{proposition}

Using Bessonov's description, one can explicitly characterize the Clark measures $\sigma_\alpha$ associated with one-component inner functions. Specifically, $\sigma_\alpha$ is a purely atomic measure given by
\[
\sigma_\alpha = \sum_{n\in \mathbb N} \sigma_\alpha(\{\zeta_n\}) \, \delta_{\zeta_n},
\]
where the boundary points $\zeta_n$ anchoring the atoms satisfy \[\angle\lim_{z \to \zeta_n} u(z) = \alpha\] and
\begin{equation}\label{E:formule-uprime}
u'(\zeta_n) = \angle\lim_{z\to \zeta_n }\frac{u(z)-u(\zeta_n)}{z-\zeta_n} = \frac{\alpha \overline{\zeta_n}}{\sigma_\alpha(\{\zeta_n\})}.
\end{equation}
In particular, $\sigma_\alpha(\{\zeta_n\}) = \frac{1}{|u'(\zeta_n)|}$. For each point $\zeta_n$, we consider the associated normalized boundary kernel introduced in Proposition \ref{T:thm724model},
\[
\widetilde{k_n}(z) := \frac{k_{\zeta_n}^u(z)}{\|k_{\zeta_n}^u\|_{H^2}} = \frac{1}{\|k_{\zeta_n}^u\|_{H^2}} \frac{1-\overline{u(\zeta_n)}u(z)}{1-\overline{\zeta_n}z}, \qquad z \in \mathbb{D}.
\]
These boundary kernels form an orthonormal basis for the model space $K_u$. The model spaces associated with one-component inner functions have been extensively studied \cite{BARANOV2005116, Volberg1998}.

\section{Spectrum of $\|\text{Q}_\zeta^u\|$} \label{S:Lowerestimates}
In this section, we describe the spectrum of the operator $\text{Q}_\zeta^u$ and provide lower estimates for the norm $\|\text{Q}_\zeta^u\|$ when $\zeta \in \mathbb{T} \setminus \sigma(u)$. We begin by proving Theorem \ref{T:spectrumQ}.

\begin{proof}[Proof of Theorem \ref{T:spectrumQ}]
Let $\lambda \in \mathbb{C}$. According to \eqref{E:def-Qu},
\[
\lambda \text{I} - \text{Q}_\zeta^u = (\text{I} - \zeta X_u)^{-1}[\lambda \text{I} - (\lambda \zeta + 1) X_u].
\]
This shows that for $\lambda = -\overline{\zeta}$, the operator $\lambda \text{I} - \text{Q}_\zeta^u$ is invertible. More\-over, for $\lambda \neq -\overline{\zeta}$, it holds that
\begin{equation}\label{E:spectraQ}
\lambda \text{I} - \text{Q}_\zeta^u = (\lambda \zeta + 1)(\text{I} - \zeta X_u)^{-1}\left(\frac{\lambda}{1 + \zeta \lambda} \text{I} - X_u\right).
\end{equation}
It would be more natural to work with the operator $X_u$. However, since we want the spectrum $\sigma(u) = \sigma(S_u) = \sigma(X_u^\ast)$ to appear, we write
\[
\lambda \text{I} - \text{Q}_\zeta^u = (\lambda \zeta + 1)(\text{I} - \zeta X_u)^{-1}\left(\frac{\overline{\lambda}}{1 + \overline{\zeta} \overline{\lambda}} \text{I} - S_u\right)^\ast.
\]
Hence,
\[
\sigma(\text{Q}_\zeta^u) = \left\{ \lambda \in \mathbb{C} \setminus \{-\overline{\zeta}\} \colon \frac{\overline{\lambda}}{1 + \overline{\zeta} \overline{\lambda}} \in \sigma(u) \right\}.
\]
One can easily check that
\[
\frac{\overline{\lambda}}{1 + \overline{\zeta} \overline{\lambda}} = \eta \iff \overline{\lambda} = \frac{\eta}{1 - \overline{\zeta} \eta},
\]
and thus the identity \eqref{E:spectrumQ} follows.

The lower estimate \eqref{E:spectrumQest} stems from the basic fact that each point in the spectrum of an operator has modulus less than or equal to the operator norm. In particular, for every $\eta \in \sigma(u) \cap \mathbb{T}$, we have that
\[
\|\text{Q}_\zeta^u\| \geq \left| \frac{\overline{\eta}}{1-\zeta \overline{\eta}} \right| = \frac{1}{|\zeta - \eta|},
\]
and we conclude by taking the supremum with respect to $\eta$. Notice that if the intersection $\sigma(u) \cap \mathbb{T}$ were empty, then $\operatorname{dist}(\zeta, \sigma(u) \cap \mathbb{T}) = \inf \emptyset = +\infty$, making the statement trivial.

For the rest of the proof, notice that by assumption $\zeta \notin \sigma(u)$, therefore the denominators appearing in \eqref{E:spectraQz1} and \eqref{E:spectraQz2} do not vanish. From \eqref{E:spectraQ}, we have the identity
\[
\text{Ker}(\lambda \text{I} - \text{Q}_\zeta^u) = \text{Ker}\left(\frac{\lambda}{1+\zeta\lambda} \text{I} - X_u\right),
\]
and thus \eqref{E:spectraQz1} follows directly from \eqref{E:spectraXu1}.

Additionally, by \eqref{E:spectraQ}, we obtain
\[
\text{Ker}\left((\lambda \text{I} - \text{Q}_\zeta^u)^*\right) = \left(\text{I} - \zeta X_u\right)^* \text{Ker}\left(\left(\frac{\lambda}{1+\zeta\lambda} \text{I} - X_u\right)^*\right),
\]
and
\[
\text{Ran}\left((\lambda \text{I} - \text{Q}_\zeta^u)^*\right) = \left(\text{I} - \zeta X_u\right)^{-1} \text{Ran}\left(\left(\frac{\lambda}{1+\zeta\lambda} \text{I} - X_u\right)^*\right).
\]
Therefore, \eqref{E:spectraQz2} also follows from the invertibility of $\left(\text{I} - \zeta X_u\right)$ and \eqref{E:spectraXu2}.
\end{proof}

Theorem \ref{T:spectrumQ} shows in a quantitative way that when a point $\zeta \in \mathbb{T} \setminus \sigma(u)$ approaches the spectrum of $u$ on an arc of the boundary $\mathbb{T}$, the operator norm $\|\text{Q}_\zeta^u\|$ becomes unbounded on that arc.

\section{A lower estimate for $\|\text{Q}_\zeta^u\|$} \label{S:Lowerestimates2}

The following lemma describes the pointwise derivative $f'(\zeta)$ as a bounded linear functional on $K_u$.

\begin{lemma}\label{L:boundaryderiv}
Let $u$ be an inner function, and let $\zeta \in \mathbb{T}\setminus \sigma(u)$. Then, for every $f\in K_u$, the derivative $f'(\zeta)$ exists and the linear mapping
\[
\begin{array}{cccc}
\mathfrak{D}_\zeta: & K_u & \longrightarrow & \mathbb{C}\\
& f & \longmapsto & f'(\zeta)
\end{array}
\]
is a bounded linear functional on $K_u$ with norm
\begin{equation}\label{L:boundaryderivnorm1}
\|\mathfrak{D}_\zeta\|_{(K_u)^*} = \|\text{Q}_\zeta^u k_\zeta^u\|_{H^2}.
\end{equation}
\end{lemma}

\begin{proof}
Since $\zeta\in\mathbb{T}\setminus \sigma(u)$, there exists $\delta>0$ such that $u$ and every element in $K_u$ has an analytic extension in $B_\mathbb{C}(\zeta,\delta)$. For $w\in\mathbb{D}$ with $|w-\zeta|<\delta$ and for every $f\in K_u$, we have
\begin{eqnarray*}
\left| \bigg\langle f, \frac{k^u_{w}-k_\zeta^u}{\overline{w}-\overline{\zeta}}\bigg\rangle_{H^2} \right| &=& \left|\frac{f(w)-f(\zeta)}{w-\zeta}\right| \\
&=& \left|\frac{1}{w-\zeta} \int_{\zeta}^{w} f'(z) \,dz\right| \leq C_f < \infty.
\end{eqnarray*}
Hence, by the uniform boundedness principle, there exists a constant $C>0$ such that
\[\left\|  \frac{k^u_{w}-k_\zeta^u}{\overline{w}-\overline{\zeta}}\right\|_{H^2}\leq C,\qquad w\in\mathbb{D}\cap B_\mathbb{C}(\zeta,\delta).\]
We can then extract a sequence $(w_n)$ in $\mathbb{D}$ that tends to $\zeta$ such that the difference quotients $(\overline{w_n}-\overline{\zeta})^{-1}(k^u_{w_n}-k_\zeta^u)$ weakly converge to a function $\psi_\zeta\in K_u$. Hence, for every $f\in K_u$, one has
\[
\langle f,\psi_\zeta \rangle_{H^2} = \lim_n\bigg\langle f, \frac{k^u_{w_n}-k_\zeta^u}{\overline{w_n}-\overline{\zeta}} \bigg\rangle_{H^2} = \lim_n\frac{f(w_n)-f(\zeta)}{w_n-\zeta} = f'(\zeta).
\]
This proves that the operator $\mathfrak{D}_\zeta$ is an element of the dual $(K_u)^*$, and its norm $\|\mathfrak{D}_\zeta\|_{(K_u)^*}$ is equal to $\|\psi_\zeta\|_{H^2}$. 

To conclude the proof, it suffices to show that $\|\psi_\zeta\|_{H^2} = \|\text{Q}_\zeta^u k_\zeta^u\|_{H^2}$. Hence, for $z\in\mathbb{D}$,
\begin{align*}
\psi_\zeta(z) &= \langle \psi_\zeta, k_z^u\rangle_{H^2} = \lim_n\bigg\langle \frac{k^u_{w_n}-k_\zeta^u}{\overline{w_n}-\overline{\zeta}}, k_z^u \bigg\rangle_{H^2} \\
&= \lim_n\frac{\overline{k^u_{z}(w_n)}-\overline{k_z^u(\zeta)}}{\overline{w_n}-\overline{\zeta}} = \overline{\left(k_z^u\right)'(\zeta)}.
\end{align*}
Differentiating the kernel $k_z^u$ we obtain
\begin{eqnarray*}
\psi_\zeta(z) &=& \frac{-\overline{u'(\zeta)}u(z)(1-\overline{\zeta}z)+z\left(1-\overline{u(\zeta)}u(z)\right)}{(1-\overline{\zeta}z)^2}\\
&=& \frac{-\zeta\overline{u(\zeta)}u(z)k_{\zeta}^{u}(\zeta)+zk_{\zeta}^{u}(z)}{1-\overline{\zeta}z}, \qquad z \in \mathbb{D}.\end{eqnarray*}
But since, for $m$-a.e. $z\in\mathbb{T}$,
\[
\frac{\,\, k_{\zeta}^{u}(z) \,\,}{\overline{k_{\zeta}^{u}(z)}}  = \frac{\zeta \, \overline{u(\zeta)} }{z \, \overline{u(z)}},
\]
we have that $z k_{\zeta}^{u}(z) = \zeta  \overline{u(\zeta)} u(z) \overline{k_{\zeta}^{u}(z)}$, $m$-a.e. $z\in\mathbb{T}$, and thus the formula for $\psi_\zeta(z)$ simplifies to
\[
\psi_\zeta(z) = \overline{z}\zeta \overline{u(\zeta)} u(z) \overline{\big( Q_{\zeta}^{u} k_{\zeta}^{u} \big)(z)},
\]
for almost every $z \in \mathbb{T}$. Note the use of $k_\zeta^u(\zeta)=|u'(\zeta)|=u'(\zeta)\zeta\overline{u(\zeta)}$ from Proposition \ref{T:thm724model} in the above calculation. Therefore, $\|\psi_\zeta\|_{H^2} = \|\text{Q}_\zeta^u k_\zeta^u\|_{H^2}$.
\end{proof}

The above proof is adopted from \cite[Section 21.6]{MR3617311}. Note that, more explicitly, \eqref{L:boundaryderivnorm1} says
\begin{equation}\label{L:boundaryderivnorm}
|f'(\zeta)| \leq \|\text{Q}_\zeta^uk_\zeta^u\|_{H^2}\|f\|_{H^2}, \qquad f\in K_u,
\end{equation}
and the constant on the right side is sharp.

\begin{proof}[Proof of Theorem \ref{T:lower-estimate}]
From Lemma \ref{L:boundaryderiv} and the mere definition of norm on the dual space, we have 
\begin{equation}\label{E:test}
\|\text{Q}_\zeta^uk_\zeta^u\|_{H^2}^2 = \|\mathfrak{D}_\zeta\|_{(K_u)^*} = \sup_{\|f\|_{K_u}=1} |f'(\zeta)|^2.
\end{equation}
To obtain a lower bound, we apply \eqref{E:test} with $f = k_\zeta^u$. By direct computation, it follows that
\[
\big(k_\zeta^u\big)'(z) = \frac{\overline{u(\zeta)}}{\overline{\zeta}(\zeta-z)}\left[\frac{u(z)-u(\zeta)}{z-\zeta}-u'(z)\right],\qquad z\in\mathbb{D}.
\]
Expanding $u(z)$ and $u'(z)$ in Taylor series in a neighborhood of $\zeta$, it is easy to see that 
\begin{equation}\label{E:formule-kzprime}
(k_\zeta^u)'(\zeta) = \frac{1}{2} \zeta \, \overline{u(\zeta)} \, u''(\zeta).   
\end{equation}
In particular, \eqref{E:test} yields
\begin{align*}
\|\text{Q}_\zeta^u(k_\zeta^u)\|_{H^2}^2 \geq \frac{|(k_\zeta^u)'(\zeta)|^2}{\|k_\zeta^u\|_{H^2}^2}= \frac{|u''(\zeta)|^2}{4|u'(\zeta)|}.
\end{align*}
We conclude that
\[
\|\text{Q}_\zeta^u\|^2 \geq \frac{\|\text{Q}_\zeta^u(k_\zeta^u)\|_{H^2}^2}{\|k_\zeta^u\|_{H^2}^2} \geq \frac{1}{\|k_\zeta^u\|_{H^2}^2}\frac{|u''(\zeta)|^2}{4|u'(\zeta)|} = \frac{|u''(\zeta)|^2}{4|u'(\zeta)|^2}.
\]
\end{proof}

The following proposition provides more information on the quantity $|u''|/|u'|$ appearing in Theorem \ref{T:lower-estimate}.
\begin{proposition}\label{P:u''/u'}
Let $u$ be an inner function, and consider an open arc $\Delta\subseteq \mathbb{T} \setminus \sigma(u)$ such that $\overline{\Delta}\cap \sigma(u)\neq\emptyset$. If $|u'|$ is unbounded on $\Delta$, then so is $|u''|/|u'|$ on $\Delta$.
\end{proposition}

\begin{proof}
It easily follows from Proposition \ref{T:thm724model} that $u' \neq 0$ on $\Delta$: for every point $\zeta \in \Delta$, the modulus $|u'(\zeta)|$ is equal to the squared norm of the boundary kernel $k_\zeta^u$, which is not identically zero, since $u$, by assumption, cannot be constant. Furthermore, $u$ admits an analytic extension across $\Delta$. This implies that there exists an open set $\Omega \subset \mathbb{C}$ containing $\Delta$ on which $u$ is holomorphic and does not vanish. Without loss of generality, we can assume that $\Omega$ is simply connected, allowing the function $\log(u')$ to be well-defined and holomorphic on $\Omega$, a neighborhood of $\Delta$. Note that
\[
\frac{u''(z)}{u'(z)} = \big(\log(u')\big)'(z), \qquad z \in \Omega.
\]
Since $|u'|$ is unbounded on $\Delta$, it follows that $|\log(u')|$ is also unbounded. A standard argument using the fundamental theorem of calculus shows that $|\big(\log(u')\big)'|$ must also be unbounded on $\Delta$, concluding the proof.
\end{proof}

We note that by Proposition \ref{T:aleksandrov1}, if $u$ is a one-component inner function, then for every arc $\Delta$ as in Proposition \ref{P:u''/u'}, the quantity $|u''|/|u'|$ is indeed unbounded on $\Delta$.

\section{Upper estimates for $\|\text{Q}_\zeta^u\|$} \label{S:Upperestimates}
In what follows, we fix a one-component inner function $u$ and a point $\zeta \in \mathbb{T} \setminus \sigma(u)$. In this section, we prove Theorem \ref{T:upper-estimate}, which provides an upper estimate for the norm $\|\text{Q}_\zeta^u\|$. Since $\zeta$ is not part of the spectrum of $u$, the value $\alpha := u(\zeta) \in \mathbb{T}$ is well-defined, and we consider the corresponding Clark measure $ \sigma_\alpha $. By Proposition \ref{T:Bessonov}, we know that $\sigma_\alpha$ is a purely atomic measure. Denoting the atoms of $\sigma_\alpha$ by $(\zeta_n)_{n}$, it follows that the collection of normalized boundary kernels $\{\widetilde{k_n}\}_n$, defined as in \eqref{E:formule-uprime}, forms a complete orthonormal basis for the model space $K_u$. It is also essential to observe that our original point $\zeta \in \mathbb{T} \setminus \sigma(u)$ is necessarily an atom for $\sigma_\alpha$, as $\alpha = u(\zeta)$ by definition. The index $ \ell \in \mathbb{N} $ such that $\zeta_\ell = \zeta$ will play a significant role in the discussion below. In the following lemma, we compute the entries of the infinite matrix associated with $\text{Q}_\zeta^u$ with respect to the orthonormal basis $\{\widetilde{k_n}\}_n$.

\begin{lemma}\label{L:matrix}
Let $u$ be a one-component inner function, and let $\zeta\in\mathbb{T}\setminus\sigma(u)$. Let $[q_{ij}]$ be the infinite matrix representing the operator $\text{Q}_\zeta^u$ with respect to the orthonormal basis $\{\widetilde{k_n}\}_n$, i.e.,
\[
q_{ij} := \langle \text{Q}_\zeta^u \widetilde{k_j},\widetilde{k_i}\rangle_{H^2}, \qquad i,j \geq 1.
\]
Then the matrix $[q_{ij}]$ has the form
\[
\begin{pmatrix}
\frac{1}{\zeta_1-\zeta} & 0 & \cdots & 0 &\frac{\|k_\zeta^u\|}{\|k_{\zeta_1}^u\|} \frac{1}{\zeta-\zeta_1} & 0 & \cdots\\
0 & \frac{1}{\zeta_2-\zeta} & \cdots & 0 & \frac{\|k_\zeta^u\|}{\|k_{\zeta_2}^u\|} \frac{1}{\zeta-\zeta_2} & 0 & \cdots\\
\vdots & \vdots & \ddots & \vdots & \vdots &\vdots & \cdots \\
0 & 0 & \cdots & \frac{1}{\zeta_{\ell-1}-\zeta} & 0 & 0 & \ldots \\
\frac{(k_{\zeta_1}^u)' (\zeta)}{\|k_{\zeta_1}^u\|\|k_\zeta^u\|} & \frac{(k_{\zeta_2}^u)' (\zeta)}{\|k_{\zeta_2}^u\|\|k_\zeta^u\|} & \cdots & \frac{(k_{\zeta_{\ell-1}}^u)' (\zeta)}{\|k_{\zeta_{\ell-1}}^u\|\|k_\zeta^u\|} & \frac{(k_\zeta^u)'(\zeta)}{\|k_\zeta^u\|^2} & \frac{(k_{\zeta_{\ell+1}}^u)' (\zeta)}{\|k_{\zeta_{\ell+1}}^u\|\|k_\zeta^u\|} &\cdots \\
0 &0& \cdots & 0 & \frac{\|k_\zeta^u\|}{\|k_{\zeta_{\ell+1}}^u\|} \frac{1}{\zeta-\zeta_{\ell+1}} &  \frac{1}{\zeta_{\ell+1}-\zeta}  & \cdots\\
\vdots & \vdots & \vdots & \vdots &\vdots & \vdots & \ddots
\end{pmatrix}\!.
\]
\end{lemma}
We point out a necessary abuse of notation: in the matrix above, and only in this context, $\| \cdot\|$ represents the $H^2$-norm. 
\begin{proof}
Keep in mind that $\zeta_\ell = \zeta$, and assume $i\neq\ell$. For all $j's$, by the reproducing kernel property of $k_i$, we have
\begin{align*}
q_{ij} &= \frac{\big(\text{Q}_\zeta^u \widetilde{k_j}\big)(\zeta_i)}{\|k_{\zeta_i}^u\|_{H^2}} \\
&= \frac{1}{\|k_{\zeta_i}^u\|_{H^2}} \frac{\widetilde{k_j}(\zeta_i)-\widetilde{k_j}(\zeta)}{\zeta_i-\zeta} \\
&= \frac{1}{\|k_{\zeta_i}^u\|_{H^2}}\frac{1}{\|k_{\zeta_j}^u\|_{H^2}} \frac{\langle k_{\zeta_j}^u, k_{\zeta_i}^u\rangle_{H^2}-\langle k_{\zeta_j}^u, k_{\zeta_{\ell}}^u\rangle_{H^2}}{\zeta_i-\zeta}.
\end{align*}
We distinguish three cases and we use the orthogonality relations. If $j=i$, then $\langle k_{\zeta_j}^u, k_{\zeta_i}^u\rangle_{H^2}=\|k_{\zeta_i}^u\|_{H^2}^2$ and $\langle k_{\zeta_j}^u, k_{\zeta_\ell}^u\rangle_{H^2}=0$, so that
\[
i\neq \ell, \quad j=i \implies q_{ii}=\frac{1}{\zeta_i-\zeta}.
\]
If $j=\ell$, then $\langle k_{\zeta_j}^u, k_{\zeta_i}^u\rangle_{H^2}=0$ and $\langle k_{\zeta_j}^u, k_{\zeta_\ell}^u\rangle_{H^2}=\|k_{\zeta_\ell}^u\|_{H^2}^2$, so that
\[
i\neq \ell, \quad j=\ell \implies q_{i\ell}=-\frac{\|k_\zeta^u\|_{H^2}}{\|k_{\zeta_i}^u\|_{H^2}} \frac{1}{\zeta_i-\zeta}.
\]
If $j\neq \ell, j\neq i$, then both inner products involved equal $0$, and $q_{ij}=0$.

Now, we compute to the $\ell$-th row, i.e., $i=\ell$. This is more delicate because we are evaluating $\text{Q}_\zeta^u \widetilde{k_j}$, the difference quotient at $\zeta$, at the point $\zeta$. Since $\zeta\in\mathbb{T}\setminus\sigma(u)$, every function in the model space $K_u$ admits an analytic extension in a neighborhood of $\zeta$. In particular, for all $j's$,
we can write the Taylor series
\[
k_{\zeta_j}^u(z) = \sum_{m=0}^\infty \frac{\big(k_{\zeta_j}^u\big)^{(m)}(\zeta)}{m!}(z-\zeta)^m,
\]
where $\big(k_{\zeta_j}^u\big)^{(m)}$ is the $m^\text{th}$ derivative of $k_{\zeta_j}^u$ and $\big(k_{\zeta_j}^u\big)^{(0)} = k_{\zeta_j}^u(\zeta)$. Hence,
\begin{align*}
\big(\text{Q}_\zeta^u \, k_{\zeta_j}^u\big)(z) &= \frac{k_{\zeta_j}^u(z)-k_{\zeta_j}^u(\zeta)}{z-\zeta}\\
&=\frac{1}{z-\zeta}\bigg[ \sum_{m=0}^\infty \frac{\big(k_{\zeta_j}^u\big)^{(m)}(\zeta)}{m!}(z-\zeta)^m - k_{\zeta_j}^u(\zeta) \bigg] \\
&= \sum_{m=1}^\infty \frac{\big(k_{\zeta_j}^u\big)^{(m)}(\zeta)}{m!}(z-\zeta)^{m-1}.
\end{align*}
This implies $\big(\text{Q}_\zeta^u \, k_{\zeta_j}^u\big)(\zeta) = \big(k_{\zeta_j}^u\big)' (\zeta)$, and thus
\[ 
q_{\ell j} = \frac{\big(\text{Q}_\zeta^u \widetilde{k_j}\big)(\zeta_\ell)}{\|k_{\zeta_\ell}^u\|_{H^2}} = \frac{\big(\text{Q}_\zeta^u \, k_{\zeta_j}^u\big)(\zeta)}{\|k_{\zeta_j}^u\|_{H^2}\|k_{\zeta_\ell}^u\|_{H^2}} = \frac{\big(k_{\zeta_j}^u\big)' (\zeta)}{\|k_{\zeta_j}^u\|_{H^2}\|k_\zeta^u\|_{H^2}}.
\]
In particular, for $j=\ell$,
\[
q_{\ell \ell}\ = \frac{\big(k_\zeta^u\big)'(\zeta)}{\|k_\zeta^u\|_{H^2}^2}.
\]
\end{proof}

Note that, according to Lemma \ref{L:matrix}, the non-zero coefficients of the matrix $[q_{ij}]$ are either on the diagonal, or on the $\ell$-th row, or on the $\ell$-column, where $\ell$ is the index such that $\zeta_\ell=\zeta$. The coefficients $q_{ij} =\langle \text{Q}_\zeta^u \widetilde{k_j},\widetilde{k_i}\rangle_{H^2}$ are exploited in the proof of Theorem \ref{T:upper-estimate}.

\begin{proof}[Proof of Theorem \ref{T:upper-estimate}]
We can decompose any $f\in K_u$ using the complete orthonormal basis of normalized boundary kernels $\{\widetilde{k_n}\}$ as $f=\sum_{i}\gamma_i \widetilde{k_i}$. The coefficients $\gamma_i:=\langle f,\widetilde{k_i}\rangle_{H^2}$ satisfy the Parseval identity $\sum_i |\gamma_i|^2 = \|f\|_{H^2}^2$. By linearity, 
\begin{equation}\label{E:Tupper-estimateeq1}
\|\text{Q}_\zeta^uf\|_{H^2}\leq  |\gamma_\ell| \|\text{Q}_\zeta^u \widetilde{k_\ell}\|_{H^2}+ \left\|\text{Q}_\zeta^u\Big(\sum_{i\neq \ell}\gamma_i \widetilde{k_i} \Big)\right\|_{H^2}.
\end{equation}
We estimate each summand of \eqref{E:Tupper-estimateeq1} separately. 

Using the orthonormal basis $\{\widetilde{k_n}\}$ and coefficients $q_{ij}$, the Parseval identity yields
\[ 
\|\text{Q}_\zeta^u \widetilde{k_\ell}\|_{H^2}^2 = \sum_i \left| \langle \text{Q}_\zeta^u \widetilde{k_\ell}, \, \widetilde{k_i}\rangle_{H^2} \right|^2=\sum_i |q_{i\ell}|^2.
\]
Hence, by Lemma \ref{L:matrix}, 
\begin{equation}\label{estimateQzukl1} 
\|\text{Q}_\zeta^u\widetilde{k_\ell}\|_{H^2}^2 = \sum_i |q_{i\ell}|^2 = \frac{|(k_\zeta^u\big)'(\zeta)|^2}{\|k_\zeta^u\|_{H^2}^4} + \sum_{i\neq \ell} \frac{\|k_\zeta^u\|_{H^2}^2}{\|k_{\zeta_i}^u\|_{H^2}^2} \frac{1}{|\zeta_i-\zeta|^2}.
\end{equation}
By \eqref{E:formule-kzprime} and the Proposition \ref{T:aleksandrov1}[(ii)], it follows that
\begin{equation}\label{EQ:2der}
\frac{|(k_\zeta^u\big)'(\zeta)|^2}{\|k_\zeta^u\|_{H^2}^4}=\frac{|u''(\zeta)|^2}{4|u'(\zeta)|^2} \leq C |u'(\zeta)|^2,
\end{equation}
where $C$ is a positive constant that depends only on $u$. We can rewrite the second summand in \eqref{estimateQzukl1} as
\begin{align*}
\sum_{i\neq \ell} \frac{\|k_\zeta^u\|_{H^2}^2}{\|k_{\zeta_i}^u\|_{H^2}^2} \frac{1}{|\zeta_i-\zeta|^2} =& |u'(\zeta)| \sum_{i\neq \ell}  \frac{1}{|\zeta_i-\zeta|^2}\sigma_\alpha(\{\zeta_i\})\\ =& |u'(\zeta)|\int_{\mathbb{T}\setminus \{\zeta\}} \frac{1}{|\lambda-\zeta|^2}\,d\sigma_\alpha(\lambda).
\end{align*}
Now, by \cite[Equation $(16)$]{bessonov2015duality},
\[\int_{\mathbb{T}\setminus \{\zeta\}} \frac{1}{|\lambda-\zeta|^2}\,d\sigma_\alpha(\lambda) \leq \frac{C}{\sigma_\alpha(\{\zeta\})},\]
and we conclude that
\begin{equation}\label{Eq:qzell}
\|\text{Q}_\zeta^u(\widetilde{k_\ell})\|_{H^2}^2 \leq C |u'(\zeta)|^2.
\end{equation}
Note that we use the same symbol $C$ for a positive constant that depends only on $u$, although the various constants involved may be different from term to term. We have obtained the desired estimate for the first of the two summands of \eqref{E:Tupper-estimateeq1}, since
\[
|\gamma_\ell| \|\text{Q}_\zeta^u(\widetilde{k_\ell})\|_{H^2} \leq \left(\sum_i |\gamma_i|^2\right)^{\frac{1}{2}}\|\text{Q}_\zeta^u(\widetilde{k_\ell})\|_{H^2} \leq C |u'(\zeta)| \|f\|_{H^2}.
\]

To conclude the proof, we exhibit a similar estimate for the second summand in \eqref{E:Tupper-estimateeq1}. Notice that, now, the indexes $i$ involved are different from $\ell$. For $i\neq\ell$, by orthogonal decomposition,
\[
\text{Q}_\zeta^u \widetilde{k_i}= \sum_j \langle \text{Q}_\zeta^u \widetilde{k_i}, \widetilde{k_j}\rangle_{H^2} \, \widetilde{k_j} = \sum_j q_{ji} \widetilde{k_j} = q_{ii} \widetilde{k_i} + q_{\ell i}\widetilde{k_\ell},
\]
since, by Lemma \ref{L:matrix}, all the other coefficients $q_{ji}$ equal 0. Therefore, by linearity and once more by Lemma \ref{L:matrix},
\begin{align*}
\left\|\text{Q}_\zeta^u\Big(\sum_{i\neq \ell}\gamma_i \widetilde{k_i} \Big)\right\|^2_{H^2}
&= \left\|\sum_{i\neq \ell}\gamma_i \text{Q}_\zeta^u \widetilde{k_i} \right\|^2_{H^2} \\
&=\left\|\sum_{i\neq \ell}\gamma_i \left( q_{ii} \widetilde{k_i} + q_{\ell i}\widetilde{k_\ell} \right)\right\|^2_{H^2} \\
&= \left\|\sum_{i\neq \ell}\gamma_i\left(\frac{1}{\zeta_i-\zeta}\, \widetilde{k_i} + \frac{\big(k_{\zeta_i}^u\big)' (\zeta)}{\|k_{\zeta_i}^u\|_{H^2}\|k_\zeta^u\|_{H^2}} \widetilde{k_\ell}\right) \right\|^2_{H^2} \\
&= \left\|\sum_{i\neq \ell} \frac{\gamma_i}{\zeta_i-\zeta} \, \widetilde{k_i} + \frac{1}{\|k_\zeta^u\|_{H^2}}\left(\sum_{i\neq\ell}\gamma_i\frac{\big(k_{\zeta_i}^u\big)' (\zeta)}{\|k_{\zeta_i}^u\|_{H^2}} \right)\widetilde{k_\ell} \right\|^2_{H^2}\!\!\!.
\end{align*}
Since $\{\widetilde{k_n}\}$ is an orthonormal basis, the previous squared norm can be rewritten as the sum of squares
\begin{equation}\label{E:Tupper-estimateeq2}
\left\|\text{Q}_\zeta^u\Big(\sum_{i\neq \ell}\gamma_i \widetilde{k_i} \Big)\right\|^2_{H^2}= \sum_{i\neq \ell} \frac{|\gamma_i|^2}{|\zeta_i-\zeta|^2}+ \frac{1}{\|k_\zeta^u\|_{H^2}^2}\left|\sum_{i\neq\ell}\gamma_i\frac{\big(k_{\zeta_i}^u\big)' (\zeta)}{\|k_{\zeta_i}^u\|_{H^2}} \right|^2.
\end{equation}
In this last equation, we have rewritten the second and last summand of the original \eqref{E:Tupper-estimateeq1} in a more suitable form. Once again, we work separately on the two summands that appear in \eqref{E:Tupper-estimateeq2}. 

For the first term, we have
\[\sum_{i\neq \ell} \frac{|\gamma_i|^2}{|\zeta_i-\zeta|^2} \leq  \sup_{i\neq \ell} \frac{1}{|\zeta_i-\zeta|^2}\sum_i |\gamma_i|^2 = \frac{\|f\|_{H^2}^2}{\left(\inf_{i\neq \ell}|\zeta_i-\zeta|\right)^2}.\]
By assumption, $u$ is a one-component inner function, and this gives us information about the atoms. In particular, the atom $\zeta_i$ that is closest to $\zeta$ is one of the neighbors $\zeta_\pm$, so that
\[
\inf_{i\neq \ell}|\zeta_i-\zeta| = \min\{|\zeta-\zeta_+|,|\zeta-\zeta_-|\}.
\]
By Proposition \ref{T:Bessonov}, we conclude that
\[
\sum_{i\neq \ell} \frac{|\gamma_i|^2}{|\zeta_i-\zeta|^2} \leq \frac{C}{\sigma_\alpha(\{\zeta\})^2} \|f\|_{H^2}^2 = C |u'(\zeta)|^2\|f\|_{H^2}^2.
\]
To estimate the second summand of \eqref{E:Tupper-estimateeq2}, by adding and subtracting the same term and using Young's inequality, we obtain
\[
\left|\sum_{i\neq\ell}\gamma_i\frac{\big(k_{\zeta_i}^u\big)' (\zeta)}{\|k_{\zeta_i}^u\|_{H^2}} \right|^2 \leq2\left|\sum_{i}\gamma_i\frac{\big(k_{\zeta_i}^u\big)' (\zeta)}{\|k_{\zeta_i}^u\|_{H^2}}\right|^2 +2|\gamma_\ell|^2 \frac{\big|\big(k_{\zeta_\ell}^u\big)'(\zeta)\big|^2}{\|k_{\zeta_\ell}^u\|_{H^2}^2} .
\]
By Lemma \ref{L:boundaryderiv}, since the operator $\mathfrak{D}_\zeta \colon f\mapsto f'(\zeta)$ is linear and bounded on $K_u$, we have
\[
\sum_i \gamma_i \frac{\big(k_{\zeta_i}^u\big)'(\zeta)}{\|k_{\zeta_i}^u\|_{H^2}} 
= \sum_i \gamma_i  \frac{\mathfrak{D}_\zeta(k_{\zeta_i}^u)}{\|k_{\zeta_i}^u\|_{H^2}} 
= \mathfrak{D}_\zeta\left(\sum_i \gamma_i \widetilde{k_i}\right)
= \mathfrak{D}_\zeta (f)
= f'(\zeta).
\]
Notice that, by definition of the index $\ell$, $\widetilde{k_\ell}$ is the normalization of $k_\zeta^u$. By \eqref{L:boundaryderivnorm} and \eqref{Eq:qzell}, we have that

\begin{align*}
    |f'(\zeta)|^2 &\leq \|f\|_{H^2}^2\|\text{Q}_\zeta^u(k_\zeta^u)\|_{H^2}^2 \\
    &= \|f\|_{H^2}^2 \|k_\zeta^u\|_{H^2}^2\|\text{Q}_\zeta^u(\widetilde{k_\ell})\|_{H^2}^2\\ 
    & \leq C \|f\|_{H^2}^2 |u'(\zeta)|^3.
\end{align*}

This shows that
\[
\left|\sum_{i\neq\ell}\gamma_i\frac{\big(k_{\zeta_i}^u\big)' (\zeta)}{\|k_{\zeta_i}^u\|_{H^2}} \right|^2 \leq2C_u\|f\|_{H^2}^2\left|u'(\zeta)\right|^3 +2|\gamma_\ell|^2 \frac{\big|\big(k_{\zeta}^u\big)'(\zeta)\big|^2}{\|k_{\zeta}^u\|_{H^2}^2}.
\]
For the last term in the previous sum, as in \eqref{EQ:2der},
\[
|\gamma_\ell|^2 \frac{\big|\big(k_{\zeta}^u\big)' (\zeta)\big|^2}{\|k_{\zeta}^u\|_{H^2}^2} \leq C_u |\gamma_\ell|^2 |u'(\zeta)|^3 \leq C_u \|f\|_{H^2}^2 |u'(\zeta)|^3.
\]
Therefore, we can conclude that, for \eqref{E:Tupper-estimateeq2},
\begin{align*}
\notag \left\|\text{Q}_\zeta^u\Big(\sum_{i\neq \ell}\gamma_i \widetilde{k_i} \Big)\right\|^2_{H^2}& = \sum_{i\neq \ell} \frac{|\gamma_i|^2}{|\zeta_i-\zeta|^2}+ \frac{1}{\|k_\zeta^u\|_{H^2}^2}\left|\sum_{i\neq\ell}\gamma_i\frac{\big(k_{\zeta_i}^u\big)' (\zeta)}{\|k_{\zeta_i}^u\|_{H^2}} \right|^2  \\
&\leq C_u|u'(\zeta)|^2\|f\|_{H^2}^2 + \frac{2}{|u'(\zeta)|}\left(C_u \|f\|_{H^2}^2 |u'(\zeta)|^3\right) \\
&\leq C_u|u'(\zeta)|^2\|f\|_{H^2}^2.
\end{align*}
\end{proof}

\section{Proof of Theorem \ref{T:characterization}} \label{S:Example0}
If $u$ is a one-component inner function, then by Proposition \ref{T:aleksandrov1}, the Lebesgue measure of $\sigma(u)$ is zero. Furthermore, by Theorem \ref{T:upper-estimate}, the inequality \eqref{E:upper-estimate} holds.

Conversely, comparing \eqref{E:upper-estimate} with \eqref{E:spectrumQest}, we obtain the estimate
\[
|u'(\zeta)| \geq \frac{C_u}{\operatorname{dist}(\zeta, \sigma(u) \cap \mathbb{T})}, \qquad \zeta \in \mathbb{T} \setminus \sigma(u),
\]
which shows that $|u'|$ is unbounded on any open arc $\Delta \subseteq \mathbb{T} \setminus \sigma(u)$ such that $\overline{\Delta} \cap \sigma(u) \neq \emptyset$.

Additionally, comparing \eqref{E:upper-estimate} with \eqref{E:lower-estimate}, we deduce the bound 
\[
|u''(\zeta)| \leq C_u |u'(\zeta)|^2, \qquad \zeta \in \mathbb{T} \setminus \sigma(u),
\]
where $C_u$ is a positive constant that depends only on $u$. By the characterization given in Proposition \ref{T:aleksandrov1}, we conclude that $u$ must be one-component.

\section{An example} \label{S:Example}
We present an example showing that the lower estimate in Theorem \ref{T:lower-estimate} and the upper estimate in Theorem \ref{T:upper-estimate} are both optimal. Note that this example also shows that the exponents of the derivatives of $u$ involved in Theorems \ref{T:lower-estimate} and \ref{T:upper-estimate} are sharp.

\begin{example}
{\rm Let $u$ be the singular inner function associated with the measure $\delta_1$, i.e.,
\[ 
u(z) = \exp{\left(\frac{z+1}{z-1}\right)}, \qquad z\in\mathbb{D}.
\]
The function $u$ is a one-component inner function and its spectrum is the singleton $\{1\}$. According to Theorem \ref{T:upper-estimate}, for every $\zeta\neq1$,
\[
\|\text{Q}_\zeta^{u}\| \leq C |u'(\zeta)| = \frac{2C}{|\zeta-1|^2}.
\]
On the other hand, by directly computing the second derivative of $u$, we see that
\[
u''(z) = \frac{4z\,u(z)}{(z-1)^4}, \qquad z\in\mathbb{D}.
\]
Hence, by Theorem \ref{T:lower-estimate},
\[
\|\text{Q}_\zeta^{u}\|\geq \frac{1}{2}\frac{4}{|\zeta-1|^4}\frac{|\zeta-1|^2}{2}=\frac{1}{|\zeta-1|^2}, \qquad \zeta\neq1.
\]
Therefore, we have
\[
\|\text{Q}_\zeta^{u} \| \asymp \frac{1}{|\zeta-1|^2}, \qquad \zeta\neq1.
\]}
\end{example}

\section{Concluding remarks} \label{S:Example2}
Given a point $\zeta \in \mathbb{T}$, the \emph{local Dirichlet integral} is defined as
\[
\mathcal{D}_\zeta(f) := \frac{1}{\pi} \int_\mathbb{D} |f'(z)|^2 \frac{1 - |z|^2}{|z - \zeta|^2} \, dA(z), \qquad f \in \text{Hol}(\mathbb{D}),
\]
where $A$ denotes the two-dimensional Lebesgue measure. We define the \emph{local Dirichlet space}, denoted by $\mathcal{D}_\zeta$, as the space of holomorphic functions $f$ such that $\mathcal{D}_\zeta(f) < \infty$. This is a Hilbert space of holomorphic functions with respect to the norm
\[
\|f\|_{\mathcal{D}_\zeta}^2 = \|f\|_{H^2}^2 + \mathcal{D}_\zeta(f),
\]
and it serves as an example of a more general class of spaces known as \emph{harmonically weighted} Dirichlet spaces. These spaces were introduced by Stefan Richter in 1991 \cite{Richter1991} to study cyclic analytic two-isometries. The local Dirichlet integral is also given by the important formula
\begin{equation}\label{E:Jdouglas}
\mathcal{D}_\zeta(f) = \int_{\mathbb{T}} \bigg| \frac{f(\lambda) - f(\zeta)}{\lambda - \zeta} \bigg|^2 \, dm(\lambda), \qquad f \in \text{Hol}(\mathbb{D}),
\end{equation}
where $m$ denotes the normalized Lebesgue measure on $\mathbb{T}$. For an in-depth discussion of these spaces, see \cite{MR3185375}.

\begin{enumerate}[(i)]
\item According to \eqref{E:Jdouglas}, for an inner function $u$ and a point $\zeta \in \mathbb{T} \setminus \sigma(u)$, we have
\[
\mathcal{D}_\zeta(f) = \|\text{Q}_\zeta^u(f)\|_{H^2}^2, \qquad f \in K_u.
\]
Thus, the condition $\zeta \in \mathbb{T} \setminus \sigma(u)$ is sufficient for the inclusion $K_u \subset \mathcal{D}_\zeta$ to hold. In \cite{bellavita2023embedding}, the authors showed that $\zeta \in \mathbb{T} \setminus \sigma(u)$ is also a necessary condition for the inclusion $K_u \subset \mathcal{D}_\zeta$. Moreover, they proved that this yields a bounded embedding $\iota\colon K_u \hookrightarrow \mathcal{D}_\zeta$, with the norm of the embedding computed via the norm of $\text{Q}_\zeta^u$ as
\[
\|\iota\|_{K_u \hookrightarrow \mathcal{D}_\zeta}^2 = 1 + \|\text{Q}_\zeta^u\|^2.
\]

\item The norm of the operator $\text{Q}_\zeta^u$ is also related to  the norm of $\text{R}(\zeta,S_u):=(\zeta I-S_u)^{-1}$, the resolvent operator of the compressed shift. More explicitly, for  $\zeta \in \mathbb{T}\setminus \sigma(u)$, we have
\begin{equation} \label{E:resaysmp}
    \frac{1}{2}\left(\|\text{Q}_\zeta^u\|^2 + |u'(\zeta)|\right)\leq \|\text{R}(\zeta,S_u)\|^2 \leq \|\text{Q}_\zeta^u\|^2 + |u'(\zeta)|.
\end{equation}
In fact, since $X_u^*=S_u$, 
\[
\|\text{R}(\zeta,S_u)\|= \|\text{R}(\overline{\zeta},X_u)\|= \|(\overline{\zeta}\text{I}-X_u)^{-1}\| =\|(\text{I}-\zeta X_u)^{-1}\|.
\]
From the definition \eqref{E:def-Qu} of the operator $\text{Q}_\zeta^u$, it follows that
\[
(\text{I}-\zeta X_u)^{-1} = \text{I} + \zeta \text{Q}_\zeta^u.
\]
Thus, we have 
\[
(\text{I}-\zeta X_u)^{-1}f(z) = f(z) + \zeta \frac{f(z)-f(\zeta)}{z-\zeta} = \frac{zf(z)-\zeta f(\zeta)}{z-\zeta}.
\]
Computing the norm, we obtain 
\[
\|(\text{I}-\zeta X_u)^{-1}f\|_{H^2}^2 = \int_\mathbb{T} \bigg| \frac{\lambda f(\lambda)-\zeta f(\zeta)}{\lambda-\zeta} \bigg|^2 dm(\lambda) = \mathcal{D}_\zeta(Sf),
\]
where $S$ denotes the shift operator. However, we know that 
\[
\mathcal{D}_\zeta(Sf) = \mathcal{D}_\zeta(f)+|f(\zeta)|^2 = \|\text{Q}_\zeta^u f\|^2 + |f(\zeta)|^2.
\]
See \cite[Theorem 8.1.2]{MR3185375}. Hence, by Proposition \ref{T:thm724model}, we conclude that
\[ 
\|\text{R}(\zeta,S_u)\|^2 \leq \|\text{Q}_\zeta^u\|^2 + \|k_\zeta^u\|_{H^2}^2 = \|\text{Q}_\zeta^u\|^2 + |u'(\zeta)|.
\]
On the other hand,
\[ 
\|\text{R}(\zeta,S_u)\|^2 \geq \max\{\|\text{Q}_\zeta^u\|^2, |u'(\zeta)|\},
\]
and \eqref{E:resaysmp} follows at once.
\end{enumerate}

\bibliographystyle{plain}
\bibliography{Model-References}

\begin{thebibliography}{10}

\bibitem{Aleksandrov2000}
Alexei Aleksandrov.
\newblock On embedding theorems for coinvariant subspaces of the shift
  operator. {II}.
\newblock {\em Journal of Mathematical Sciences}, 110:2907--2929, 2000.

\bibitem{BARANOV2005116}
Anton~D. Baranov.
\newblock {B}ernstein-type inequalities for shift-coinvariant subspaces and
  their applications to {C}arleson embeddings.
\newblock {\em Journal of Functional Analysis}, 223(1):116--146, 2005.

\bibitem{bellavita2023embedding}
Carlo Bellavita and Eugenio~A. Dellepiane.
\newblock Embedding model and de {B}ranges-{R}ovnyak spaces in {D}irichlet
  spaces.
\newblock {\em Complex Variables and Elliptic Equations}, 0(0):1--19, 2024.

\bibitem{Bellavita2022O}
Carlo Bellavita and Artur Nicolau.
\newblock One component bounded functions.
\newblock {\em Computational Methods and Function Theory}, pages 1--28, 2022.

\bibitem{bessonov2015duality}
Roman~V. Bessonov.
\newblock Duality theorems for coinvariant subspaces of ${H^1}$.
\newblock {\em Advances in Mathematics}, 271:62--90, 2015.

\bibitem{Cima2020}
Joseph~A. Cima and Raymond Mortini.
\newblock {\em One-Component Inner Functions II}, pages 39--49.
\newblock Springer International Publishing, 2020.

\bibitem{cima2000backward}
Joseph~A. Cima and William~T. Ross.
\newblock {\em The Backward Shift on the {H}ardy Space}.
\newblock Number v. 79 in Mathematical surveys and monographs. American
  Mathematical Society, 2000.

\bibitem{MR3185375}
Omar El-Fallah, Karim Kellay, Javad Mashreghi, and Thomas Ransford.
\newblock {\em A primer on the {D}irichlet space}, volume 203 of {\em Cambridge
  Tracts in Mathematics}.
\newblock Cambridge University Press, Cambridge, 2014.

\bibitem{MR3617311}
Emmanuel Fricain and Javad Mashreghi.
\newblock {\em The theory of {$\mathcal{H}(b)$} spaces. {V}ol. 2}, volume~21 of
  {\em New Mathematical Monographs}.
\newblock Cambridge University Press, Cambridge, 2016.

\bibitem{MR3526203}
Stephan~R. Garcia, Javad Mashreghi, and William~T. Ross.
\newblock {\em Introduction to model spaces and their operators}, volume 148 of
  {\em Cambridge Studies in Advanced Mathematics}.
\newblock Cambridge University Press, Cambridge, 2016.

\bibitem{MR0113142}
Mikhail~S. Liv\v{s}ic.
\newblock On a class of linear operators in {H}ilbert space.
\newblock {\em American Mathematical Society Translations: Series 2},
  13:61--83, 1960.

\bibitem{MR0150592}
James~W. Moeller.
\newblock On the spectra of some translation invariant spaces.
\newblock {\em Journal of Mathematical Analysis and Applications}, 4:276--296,
  1962.

\bibitem{nicolau2021}
Artur Nicolau and Atte Reijonen.
\newblock A characterization of one-component inner functions.
\newblock {\em Bulletin of the London Mathematical Society}, 53(1):42--52,
  2021.

\bibitem{MR827223}
Nikolai~K. Nikolski.
\newblock {\em Treatise on the shift operator}, volume 273 of {\em Grundlehren
  der mathematischen Wissenschaften [Fundamental Principles of Mathematical
  Sciences]}.
\newblock Springer-Verlag, Berlin, 1986.

\bibitem{Reijonen_2019}
Atte Reijonen.
\newblock Remarks on one-component inner functions.
\newblock {\em Annales Fennici Mathematici}, 44(1):569–580, 2019.

\bibitem{Richter1991}
Stefan Richter.
\newblock A representation theorem for cyclic analytic two-isometries.
\newblock {\em Transactions of the American Mathematical Society},
  328(1):325--349, 1991.

\bibitem{Ross2013LENSLO}
William~T. Ross.
\newblock Lens lectures on {A}leksandrov-{C}lark measures.
\newblock In {\em https://api.semanticscholar.org/CorpusID:947069}, 2013.

\bibitem{saksman2007}
Eero Saksman.
\newblock An elementary introduction to {C}lark measures.
\newblock In {\em Topics in complex analysis and operator theory}, page
  85–136, Spain, 2007. Universidad de M{\'a}laga.
\newblock Winter School in Complex Analysis and Operator Theory.

\bibitem{Volberg1998}
Alexander Volberg and Sergei Treil.
\newblock Imbedding theorems for the invariant subspaces of the backward shift
  operator.
\newblock {\em Journal of Soviet Mathematics}, 42:1562--1572, 1998.

\end{thebibliography}

\end{document}